\documentclass[letterpaper, 10 pt, conference]{ieeeconf}  

\IEEEoverridecommandlockouts                              

\usepackage{graphics} 
\usepackage{graphicx}

\usepackage{times} 
\usepackage{amsmath} 
\usepackage{amssymb}  

\usepackage{tikz} 
\usepackage{pgfplots} 

\usepackage{algpseudocode}%
\usepackage{algorithm}
\usepackage{epstopdf}

\usepackage{cite}

\usepackage{color}

\usepackage{lipsum}
\usepackage{verbatim}

\usepackage{amsthm}

\usepackage{mathtools}

\newtheorem{theorem}{\bf Theorem}[section] \newtheorem{definition}{\bf Definition}[section]
\newtheorem{lemma}{\bf Lemma}[section] 
 \newtheorem{corollary}{\bf Corollary}[section] \newtheorem{proposition}{\bf Proposition}[section]

\newtheorem{claim}{\bf Claim}[section]

\usepackage{bbm}
\usepackage{bm}

\newcommand\rmc{\mathrm{c}}

\newcommand\rmo{\mathrm{o}}

\newcommand\rmr{\mathrm{r}}

\newcommand\rmu{\mathrm{u}}

\newcommand\rmx{\mathrm{x}}
\newcommand\rmy{\mathrm{y}}

\newcommand\bbi{\mathbb{I}}

\newcommand\bbr{\mathbb{R}}

\newcommand\calT{\mathcal{T}}

\title{\LARGE \bf
A quantitative and constructive proof of Willems' Fundamental Lemma and its implications
}

\author{
Julian Berberich$^1$, Andrea Iannelli$^2$, Alberto Padoan$^2$, Jeremy Coulson$^2$, Florian D\"orfler$^2$, and Frank Allg\"ower$^1$
\thanks{
F. Allgöwer is thankful that his work was funded by Deutsche Forschungsgemeinschaft (DFG, German Research Foundation) under Germany's Excellence Strategy - EXC 2075 - 390740016 and under grant 468094890. 
		F. Allgöwer acknowledges the support by the Stuttgart Center for Simulation Science (SimTech).}
\thanks{$^1$University of Stuttgart, Institute for Systems Theory and Automatic Control, 70550 Stuttgart, Germany (email:$\{$julian.berberich, frank.allgower$\}$@ist.uni-stuttgart.de)}
\thanks{$^2$Department of Information Technology and
Electrical Engineering, ETH Zurich, 8092 Zurich, Switzerland (email:$\{$iannelli,apadoan,jcoulson,doerfler$\}$@control.ee.ethz.ch)}
}

\begin{document}

\maketitle
\thispagestyle{empty}
\pagestyle{empty}

\begin{abstract}
Willems' Fundamental Lemma provides a powerful data-driven parametrization of all trajectories of a controllable linear time-invariant system based on one trajectory with persistently exciting (PE) input.
In this paper, we present a novel proof of this result which is inspired by the classical adaptive control literature and differs from existing proofs in multiple aspects. 
The proof involves a quantitative and directional PE notion, allowing to characterize robust PE properties via singular value bounds, as opposed to binary rank-based PE conditions.
Further, the proof is constructive, i.e., we derive an explicit PE lower bound for the generated data.
As a contribution of independent interest, we generalize existing PE results from the adaptive control literature and reveal a crucial role of the system's zeros.
\end{abstract}


\section{Introduction}
The Fundamental Lemma by Willems et al.~\cite{willems2005note} allows all finite-length trajectories of a controllable linear time-invariant (LTI) system to be parametrized based on one input-output trajectory with persistently exciting (PE) input.
Starting from this data-driven system parametrization, a plethora of system analysis and control methods have been developed, analyzed, and successfully applied in practice, see~\cite{markovsky2021behavioral} for a recent survey. 
Further, different generalizations of the Fundamental Lemma have been developed in recent years, e.g., to certain classes of nonlinear systems~\cite{berberich2020trajectory,rueda2020data,verhoek2021fundamental} or by relaxing some of its assumptions~\cite{yu2021controllability,mishra2021controllability,markovsky2022identifiability}.
Whereas the original result by~\cite{willems2005note} was formulated and proven in the behavioral framework, the recent paper~\cite{waarde2020willems} contains a proof in the state-space framework.
An alternative proof in the behavioral framework for the single-input case is provided in~\cite{markovsky2023persistency}.
In this paper, we present a quantitative and constructive proof of the Fundamental Lemma based on arguments from the classical literature on adaptive control~\cite{moore1983persistence,green1986persistence,astrom2008adaptive}.

PE properties and their implications on parameter convergence have been at the center of the field of adaptive control for several decades~\cite{moore1983persistence,green1986persistence,astrom2008adaptive}.
These classical PE notions are typically based on squares and summation, e.g., for some signal $\{u_k\}_{k=0}^{N-1}$, the sum $\sum_{k=0}^{N-1} u_ku_k^\top$ is positive definite.
On the contrary, the Fundamental Lemma~\cite{willems2005note} and its applications in data-driven control~\cite{markovsky2021behavioral} mainly rely on a rank-based PE notion.
That is, the Hankel matrix of depth $L$ corresponding to the sequence $\{u_k\}_{k=0}^{N-1}$ with each $u_k\in\bbr^m$,
\begin{align*}
H_L(u)=\begin{bmatrix}u_0&u_1&\dots&u_{N-L}\\
u_1&u_2&\dots&u_{N-L+1}\\
\vdots&\vdots&\ddots&\vdots\\
u_{L-1}&u_L&\dots&u_{N-1}
\end{bmatrix}\in\bbr^{mL\times(N-L+1)},
\end{align*}
needs to have full row rank.
While this rank condition is equivalent to positive definiteness of $H_L(u) H_L(u)^\top$, i.e., to a squares-and-summation-based PE notion, only the latter allows to precisely quantify the richness of the data in terms of lower bounds on the eigenvalues of $H_L(u)H_L(u)^\top$.
This quantification has several important implications which we describe in more detail later in the paper and which motivated~\cite{coulson2022robust} to derive a \emph{robust} Fundamental Lemma.
In this paper, we seek to further reconcile the behavioral approach to data-driven control with the classical adaptive control literature.
To this end, we show that the Fundamental Lemma~\cite{willems2005note} can be proved using as starting point the main result of~\cite{green1986persistence}.
In addition to providing a different viewpoint on the Fundamental Lemma, the presented proof has the following advantages:
1) it involves a quantitative (similar to~\cite{coulson2022robust}) and directional description of PE properties, which leads to more robust 
conditions;
2) it is constructive, i.e., we derive an explicit lower bound on the richness of the data matrices.
Further, open research questions for methods relying on the Fundamental Lemma concern noisy data and nonlinearities in the data-generating system.
Owing to the constructive nature of the proof and its more direct links with PE properties, we believe that the results presented here can provide the basis for further generalizations, which constitute an interesting direction for future research.

Before providing a new proof of the Fundamental Lemma (Section~\ref{sec:WFL}), we first extend the results of~\cite{green1986persistence} in Section~\ref{sec:pe_lin}, where we derive input design conditions to guarantee a PE output signal for arbitrary initial conditions and arbitrary output reachable systems.
In contrast to~\cite{green1986persistence}, our results employ a directional PE notion, leading to less conservative results and illustrating the role of the Markov parameters in a transparent fashion.
Moreover, we show that the PE requirements can be relaxed under certain conditions related to the number of zeros of the system.

\subsubsection*{Notation}
We write $\bbi_{\geq0}$ ($\bbi_{>0}$) for the set of nonnegative (positive) integers and $\bbi_{[a,b]}$ for all integers in the interval $[a,b]$.
An identity matrix of dimension $n$ is denoted by $I_n$.
We denote the image of a matrix $A$ by $\mathrm{im}(A)$.
The minimum eigenvalue and singular value of $A=A^\top$ are defined as $\lambda_{\min}(A)$ and $\sigma_{\min}(A)$, respectively.
We write $A\succ0$ ($A\succeq0$) or $A\prec0$ ($A\preceq0$) if $A$ is positive or negative (semi-)definite, respectively.
The Kronecker product of two matrices $A$ and $B$ is denoted by $A\otimes B$.
Further, we denote a window of a multivariate sequence $\{x_k\}_{k=0}^{N-1}$ by $x_{[a,b]}=\begin{bmatrix}x_a^\top&\dots&x_b^\top\end{bmatrix}^\top$ and we write $x=x_{[0,N-1]}$ for the full stacked sequence.
We write $\bm{\Sigma}_{n,m,p}^{\rmo\rmr}$ for the set of output reachable systems (see Definition~\ref{def:output_reachable}) of order $n$ with $m$ inputs and $p$ outputs, i.e.,
$\bm{\Sigma}_{n,m,p}^{\rmo\rmr}=\{(A,B,C,D)\mid A\in\bbr^{n\times n}, B\in\bbr^{n\times m}, C\in\bbr^{p\times n}, D\in\bbr^{p\times m}, (A,B,C,D)\>\text{output reachable}\}$, and similarly $\bm{\Sigma}_{n,m,p}^{\rmc}$ for the set of controllable systems.
Finally, we define $\{y_k(\Sigma,x_0,u)\}_{k=0}^{N-1}$ as the output sequence generated by applying the input $\{u_k\}_{k=0}^{N-1}$ to the system $\Sigma$ with initial condition $x_0$.

\section{Revisiting persistence of excitation}\label{sec:pe_lin}
After introducing the considered PE notion in Section~\ref{subsec:pe_lin_output}, we characterize under what conditions on the input signal the output of an LTI system is PE in Section~\ref{subsec:pe_lin_imply}.
This result will be instrumental in the new proof of the Fundamental Lemma in Section~\ref{sec:WFL}, and it generalizes the main result of~\cite{green1986persistence} by using a more general quantitative PE notion.
In Section~\ref{subsec:pe_lin_tailored}, we show that the input PE condition can be relaxed under an assumption on the Markov parameters.

\subsection{Quantitative notion of PE}\label{subsec:pe_lin_output}
Consider the discrete-time LTI system
\begin{align}\label{eq:sys_LTI}
\Sigma_0:
\begin{cases}
x_{k+1}&=Ax_k+Bu_k,\\
y_k&=Cx_k+Du_k
\end{cases}
\end{align}
with state $x_k\in\bbr^n$, input $u_k\in\bbr^m$, and output $y_k\in\bbr^p$.
Collect the first $n+1$ Markov parameters of~\eqref{eq:sys_LTI} in the matrix
\begin{align}\nonumber
\Gamma=&\begin{bmatrix}\Gamma_0&\Gamma_1&\dots&\Gamma_n\end{bmatrix}\\\label{eq:def_output_reachable}
=&\begin{bmatrix}D&CB&CAB&\dots&CA^{n-1}B\end{bmatrix}.
\end{align}
The following is a classical definition of persistence of excitation~\cite{willems2005note}.
\begin{definition}\label{def:pe_classical}
\cite{willems2005note}
The sequence $\{u_k\}_{k=0}^{N-1}$ with $u_k\in\bbr^m$ is persistently exciting (PE) of order $L$ if $\mathrm{rank}(H_L(u))=mL$.
\end{definition}
This concept of persistence of excitation has the shortcoming that it is merely \emph{qualitative} in the sense that, for a fixed order $L$, it is not possible to compare the PE levels of different signals.
This motivates the following quantitative version of persistence of excitation as recently proposed by~\cite{coulson2022robust} in the context of data-driven control and originally proposed in a similar form in the adaptive control literature~\cite{moore1983persistence,green1986persistence,astrom2008adaptive}.
\begin{definition}
\label{def:pe}
The sequence $\{u_k\}_{k=0}^{N-1}$ is $K$-PE of order $L$ for some matrix $K\succ0$ if 
\begin{align}\label{eq:pe}
H_L(u)H_L(u)^\top=\sum_{k=0}^{N-L}u_{[k,k+L-1]}u_{[k,k+L-1]}^\top\succeq K.
\end{align}
\end{definition}
Note that $\{u_k\}_{k=0}^{N-1}$ is $K$-PE of order $L$ for some $K\succ0$ if and only if $\{u_k\}_{k=0}^{N-1}$ is PE of order $L$.
However, the flexibility gained via $K$ in Definition~\ref{def:pe} allows for a more precise quantification of PE properties than Definition~\ref{def:pe_classical}.
In particular, any choice $K$ in Definition~\ref{def:pe} always reduces to the same rank condition in Definition~\ref{def:pe_classical}.

Persistence of excitation properties of system components, e.g., states, outputs, or regressors, play a crucial role in the classical literature on parameter estimation and adaptive control~\cite{moore1983persistence,green1986persistence,astrom2008adaptive}, where they can be used to guarantee convergence, as well as in the recent literature on data-driven control, see~\cite{markovsky2021behavioral} for an overview.
In contrast to the classical works~\cite{moore1983persistence,green1986persistence,astrom2008adaptive}, which  typically involve infinitely many, shifted time windows or infinite sums, Definition~\ref{def:pe} imposes a condition over one finite time interval.
As a second difference to these results and to the recently suggested PE notion by~\cite{coulson2022robust}, we propose a not necessarily diagonal lower bound $K$, which allows to capture directional PE properties.

\subsection{When does a PE input imply a PE output?}\label{subsec:pe_lin_imply}
By the Cayley-Hamilton theorem, there exists a polynomial $d(z)=\sum_{j=0}^n d_j z^{n-j}$ with
\begin{align}
d=\begin{bmatrix}d_n&\cdots&d_1&d_0\end{bmatrix}^\top\in\bbr^{n+1},\>\>d_0\neq0,
\end{align}
such that $d(A)=0$.
Without loss of generality, we assume $\lVert d\rVert_2=1$.
Further, we define 
\begin{align}
M=\begin{bmatrix}M_n&\dots&M_1&M_0\end{bmatrix}\in\bbr^{p\times m(n+1)}
\end{align}
with $M_j=\sum_{q=0}^jd_{j-q}\Gamma_q$.
The following input-output representation of system~\eqref{eq:sys_LTI} will play a crucial role, cf.\ \cite{green1986persistence}.
\begin{lemma}\label{lem:IO_rep}
\cite{green1986persistence}
Any trajectory $\{u_k,y_k\}_{k=0}^{\infty}$ of~\eqref{eq:sys_LTI} satisfies, for any $k\geq n$,
\begin{align}\label{eq:sys_LTI_IO}
M u_{[k-n,k]}=\begin{bmatrix}y_{k-n}&\dots&y_{k-1}&y_k\end{bmatrix} d.
\end{align} 
%
\end{lemma}
%
Note that, for a single-input single-output (SISO) system, $M$ and $d$ simply contain the coefficients of the transfer function.
The following alternative representations of $M$ will be useful later in the paper:
%
\begin{align}\label{eq:M_alt}
M=(d^\top\otimes I_p)\bar{\Gamma}=\Gamma\left(\bar{D}\otimes I_m\right),
\end{align}
where
\begin{align}\label{eq:D_bar_def}
\bar{D}&=
\begin{bmatrix}d_n&d_{n-1}&\dots&d_0\\
d_{n-1}&d_{n-2}&\reflectbox{$\ddots$}&0\\
\vdots&\reflectbox{$\ddots$}&\reflectbox{$\ddots$}&\vdots\\
d_0&0&\dots&0\end{bmatrix}\in\bbr^{(n+1)\times(n+1)},\\\label{eq:Gamma_bar_def}
\bar{\Gamma}&=\begin{bmatrix}\Gamma_0&0&\dots&0\\
\Gamma_1&\Gamma_0&\ddots&\vdots\\
\vdots&\ddots&\ddots&0\\\
\Gamma_n&\dots&\Gamma_1&\Gamma_0
\end{bmatrix}\in\bbr^{p(n+1)\times m(n+1)}.
\end{align}
The following result characterizes %
PE properties 
of the output of~\eqref{eq:sys_LTI} in terms of the input for arbitrary initial conditions.
\begin{theorem}\label{thm:PE_Nu_y}
If the signal $\{Mu_{[k-n,k]}\}_{k=n}^{N-1}$ is PE of order $1$, then, for any $x_0\in\bbr^n$, the sequence $\{y_k(\Sigma_0,x_0,u)\}_{k=0}^{N-1}$ is PE of order $1$.
Further, if $\{Mu_{[k-n,k]}\}_{k=n}^{N-1}$ is $K_\rmu$-PE of order $1$, then, for any $x_0\in\bbr^n$, $\{y_k(\Sigma_0,x_0,u)\}_{k=0}^{N-1}$ is $K_\rmy$-PE of order $1$ with
\begin{align}\label{eq:thm_PE_Nu_y}
K_\rmy=\frac{1}{n+1}K_\rmu.
\end{align}
\end{theorem}
\begin{proof}
Let us abbreviate $y(\Sigma_0,x_0,u)$ by $y$.
Note that
\begin{align}\label{eq:thm_PE_nu_y_proof1}
K_\rmu \preceq &\sum_{k=n}^{N-1}Mu_{[k-n,k]}u_{[k-n,k]}^\top M^\top\\\nonumber
\stackrel{\eqref{eq:sys_LTI_IO}}{=}&\sum_{k=n}^{N-1}
\begin{bmatrix}y_{k-n}&\dots&y_k\end{bmatrix} dd^\top
\begin{bmatrix}y_{k-n}&\dots&y_k\end{bmatrix}^\top\\\nonumber
\preceq&\lVert d\rVert_2^2\sum_{k=n}^{N-1}
\sum_{j=0}^n y_{k-j}y_{k-j}^\top\preceq
(n+1)\sum_{k=0}^{N-1} y_ky_k^\top,
\end{align}
using $\lVert d\rVert_2^2=1$ in the last step.
This proves both
statements.
\end{proof}
Theorem~\ref{thm:PE_Nu_y} extends~\cite[Theorem 2.1]{green1986persistence} by providing a matrix-valued lower bound~\eqref{eq:thm_PE_Nu_y} instead of a diagonal one, i.e., $K_\rmu=k_\rmu I$ for some $k_\rmu>0$.
The result yields a 
sufficient condition for designing an input which guarantees a PE output, i.e.,
$\{Mu_{[k-n,k]}\}_{k=n}^{N-1}$ needs to be PE of order $1$.
In~\cite{green1986persistence}, it is claimed that this condition is not only sufficient but also necessary for achieving PE of the output
for arbitrary initial conditions.
This is, however, not true.
In the appendix, we give a counterexample and discuss where the proof of~\cite{green1986persistence} fails.

Note that the input design condition in Theorem~\ref{thm:PE_Nu_y} involves the matrix $M$, which is typically unknown in estimation and data-driven control setups.
In order to develop conditions which only involve the input, the concept of \emph{output reachability} is essential, compare~\cite{green1986persistence}.
\begin{definition}\label{def:output_reachable}
The system $\Sigma$ is \emph{output reachable} if, for any $\hat{y}\in\bbr^p$, there exist $T\in\bbi_{\geq0}$ and an input sequence $\{u_k\}_{k=0}^T$ such that the resulting output $\{y_k(\Sigma,0,u)\}_{k=0}^T$ satisfies $y_T(\Sigma,0,u)=\hat{y}$.
\end{definition}
By writing down the explicit output evolution, one observes that system~\eqref{eq:sys_LTI} is output reachable if and only if $\Gamma$ has full row rank, that is, if and only if $M$ has full row rank~\cite[Lemma 2.1]{green1986persistence}.
Note that, if the signal $\{Mu_{[k-n,k]}\}_{k=n}^{N-1}$ is PE of order $1$, then $M$ has full row rank~\cite[Lemma 2.2]{green1986persistence}.
Thus, output reachability is necessary for ensuring a PE output via Theorem~\ref{thm:PE_Nu_y}.

The following corollary of Theorem~\ref{thm:PE_Nu_y}, which extends~\cite[Corollary 2.1]{green1986persistence}, provides a design condition for the input signal to ensure a PE output for an arbitrary output reachable system and for arbitrary initial conditions.
\begin{corollary}\label{cor:PE_u_y}
If the signal $\{u_k\}_{k=0}^{N-1}$ is PE of order $n+1$, then, for any $\Sigma\in\bm{\Sigma}_{n,m,p}^{\rmo\rmr}$ and any $x_0\in\bbr^n$,
the sequence $\{y_k(\Sigma,x_0,u)\}_{k=0}^{N-1}$ is PE of order $1$.
Further, if $\{u_k\}_{k=0}^{N-1}$ is $K_\rmu$-PE of order $n+1$, then, for any $\Sigma\in\bm{\Sigma}_{n,m,p}^{\rmo\rmr}$ and any $x_0\in\bbr^n$, $\{y_k(\Sigma,x_0,u)\}_{k=0}^{N-1}$ is $K_\rmy$-PE of order $1$ with
\begin{align}\label{eq:cor_PE_u_y}
K_\rmy=\frac{1}{n+1}MK_\rmu M^\top.
\end{align}
\end{corollary}
\begin{proof}
The first claim follows from Theorem~\ref{thm:PE_Nu_y} since $\{Mu_{[k-n,k]}\}_{k=n}^{N-1}$ is PE of order $1$ for arbitrary $M$ with full row rank if $\{u_{[k-n,k]}\}_{k=n}^{N-1}$ is PE of order $1$, which, in turn, holds if $\{u_k\}_{k=0}^{N-1}$ is PE of order $n+1$.
For the second claim, if $\{u_k\}_{k=0}^{N-1}$ is $K_\rmu$-PE of order $n+1$ then $MK_\rmu M^\top\preceq\sum_{k=n}^{N-1}Mu_{[k-n,k]}u_{[k-n,k]}^\top M^\top$, from which~\eqref{eq:cor_PE_u_y} follows by the same steps as in the proof of Theorem~\ref{thm:PE_Nu_y}.
\end{proof}
Equation~\eqref{eq:cor_PE_u_y} provides a quantitative input design condition to ensure a desired output PE level $K_\rmy\succ0$, i.e., it is sufficient to choose the input as $K_\rmu$-PE of order $n+1$ for some $K_\rmu\succ0$ satisfying $\frac{1}{n+1}MK_\rmu M^\top\succeq K_\rmy$.
Corollary~\ref{cor:PE_u_y} is conceptually also similar to the classical result~\cite[Theorem 2.9]{astrom2008adaptive}, which describes PE properties of filtered signals.
The result remains true when replacing $1$ by any integer $L\in\bbi_{>0}$, i.e., it can be shown that the output is PE of order $L$ if the input is PE of order $L+n$.
This extension is omitted for space reasons.

Corollary~\ref{cor:PE_u_y} extends~\cite[Corollary 2.1]{green1986persistence} by allowing for a matrix-valued PE lower bound, yielding insights on \emph{directional persistence of excitation}.
To be precise, the matrix in~\eqref{eq:cor_PE_u_y} does not only bound the minimum singular value of $\sum_{k=0}^{N-1}y_ky_k^\top$, but is described via a (not necessarily diagonal) matrix depending on the system parameters via $M$.
In fact, inserting $M=\Gamma(\bar{D}\otimes I_m)$ from~\eqref{eq:M_alt} in~\eqref{eq:cor_PE_u_y}, we have
\begin{align}\label{eq:K2_lower_bound_intuitive_output}
K_\rmy\succeq&\frac{\lambda_{\min}(K_\rmu)}{n+1}\sigma_{\min}(\bar{D})^2\Gamma\Gamma^\top.
\end{align}
Thus, the Markov parameters $\Gamma$ directly influence (a lower bound on) the matrix $K_\rmy$ and, in particular, directions that are easy to reach or observe are more PE.
This insight is conceptually similar to~\cite{tsiamis2021linear}, which shows that the difficulty of system identification is influenced by the controllability properties of the underlying system.

Using $M=(d^\top\otimes I_p)\bar{\Gamma}$ by~\eqref{eq:M_alt} as well as $\lVert d\rVert_2=1$, we can also derive the following alternative bound
\begin{align}\label{eq:K2_lower_bound_intuitive_input}
K_\rmy\succeq&\frac{1}{n+1}\lambda_{\min}(\bar{\Gamma}K_\rmu\bar{\Gamma}^\top)I_p.
\end{align}
Since the lower bound~\eqref{eq:K2_lower_bound_intuitive_input} is diagonal, it does not provide any directional information on the richness of the output.
However, in contrast to~\eqref{eq:K2_lower_bound_intuitive_output}, it illustrates the role of the PE directions of the input.
Specifically, the bound depends on $\lambda_{\min}(\bar{\Gamma}K_\rmu\bar{\Gamma}^\top)$, i.e., to be more effective, the input should be tailored to the row space of the Toeplitz matrix $\bar{\Gamma}$ containing the Markov parameters.

To summarize, Corollary~\ref{cor:PE_u_y} implies the bounds~\eqref{eq:K2_lower_bound_intuitive_output} and~\eqref{eq:K2_lower_bound_intuitive_input}, which allow for an intuitive interpretation in terms of directional persistence of excitation of the output and input, respectively.
While these results cannot be used to construct explicit PE bounds if $\Gamma$ and $d$ are unknown, they reveal interesting insights in the interplay of the input and output PE properties and the system's Markov parameters.

\subsection{Relaxing PE requirements}\label{subsec:pe_lin_tailored}
Corollary~\ref{cor:PE_u_y} provides an input design condition for guaranteeing a PE output for arbitrary output reachable systems.
In this section, we investigate whether this condition can be relaxed when restricting the system class by using additional prior knowledge on the system.
To this end, we consider the condition that the first $r$ Markov parameters vanish, i.e.,
\begin{align}\label{eq:cor_PE_u_y_tailored}
\Gamma_i&=0\>\>\text{for}\>\>i=0,\dots,r-1,\>\>\text{and}\>\>\Gamma_{r}\neq0.
\end{align}
\begin{corollary}\label{cor:PE_u_y_tailored}
If $\{u_k\}_{k=0}^{N-r-1}$ is PE of order $n+1-r$, then, for any $\Sigma\in\bm{\Sigma}_{n,m,p}^{\rmo\rmr}$ satisfying~\eqref{eq:cor_PE_u_y_tailored} and any $x_0\in\bbr^n$, the sequence $\{y_k(\Sigma,x_0,u)\}_{k=0}^{N-1}$ is PE of order $1$.
\end{corollary}
\begin{proof}
Note that~\eqref{eq:M_alt} and~\eqref{eq:cor_PE_u_y_tailored} imply
\begin{align}\label{eq:cor_PE_u_y_tailored_proof}
Mu_{[k-n,k]}=
&\bar{M}_ru_{[k-n,k-r]}
\end{align}
with $\bar{M}_r\coloneqq (\bar{d}_r^\top\otimes I_p)\bar{\Gamma}_r$, $\bar{d}_r\coloneqq\begin{bmatrix}d_{n-r}&\dots&d_0\end{bmatrix}^\top$, and 
\begin{align*}
\bar{\Gamma}_r\coloneqq\begin{bmatrix}\Gamma_{r}&0&\dots&0\\
\Gamma_{r+1}&\Gamma_{r}&\ddots&\vdots\\
\vdots&\ddots&\ddots&0\\
\Gamma_n&\dots&\Gamma_{r+1}&\Gamma_{r}
\end{bmatrix}.
\end{align*}
If $\{u_k\}_{k=0}^{N-r-1}$ is PE of order $n+1-r$, then $\{\bar{M}_ru_{[k-n,k-r]}\}_{k=n}^{N-1}$ is PE of order $1$ for arbitrary matrices $\bar{M}_r$ with full row rank.
Together with~\eqref{eq:cor_PE_u_y_tailored_proof}, this implies that $\{Mu_{[k-n,k]}\}_{k=n}^{N-1}$ is PE of order $1$ for arbitrary output reachable systems satisfying~\eqref{eq:cor_PE_u_y_tailored}, i.e., by Theorem~\ref{thm:PE_Nu_y}, the output is PE of order $1$.
\end{proof}
According to Corollary~\ref{cor:PE_u_y_tailored}, the PE requirements of Corollary~\ref{cor:PE_u_y} can be relaxed depending on how many Markov parameters vanish.
For SISO systems,~\eqref{eq:cor_PE_u_y_tailored} is equivalent to system~\eqref{eq:sys_LTI} having relative degree $r$ and $n-r$ zeros.
In this case, Corollary~\ref{cor:PE_u_y_tailored} means that a sufficient condition for the output being PE is that the input is PE of order $1+\text{\#zeros}$, i.e., additional zeros require a richer input signal to achieve the same PE level of the output.
Verifying condition~\eqref{eq:cor_PE_u_y_tailored} only from data is an interesting issue for future research.

Note that the lower bound~\eqref{eq:K2_lower_bound_intuitive_input} is zero if $\Gamma_0=D$ does not have full row rank.
In this case, a non-trivial bound can be derived by following the same strategy as in Corollary~\ref{cor:PE_u_y_tailored}.

Finally, the persistence of excitation condition for the input in Corollary~\ref{cor:PE_u_y_tailored} only involves the data up to time step $N-r-1$.
This is due to the fact that the time steps $k=N-r,\dots,N-1$ have no influence on the output trajectory.

\section{A quantitative and constructive proof of the Fundamental Lemma}\label{sec:WFL}
In this section, we present our main result: 
a novel proof of (a generalization of) the Fundamental Lemma based on~\cite{green1986persistence}, i.e., based on the results in Section~\ref{sec:pe_lin}.
Let us recall the Fundamental Lemma~\cite[Theorem 1]{willems2005note} in the context of state-space systems, adjusted to our notation.
\begin{theorem}\cite{willems2005note}\label{thm:willems}
If $\{u_k\}_{k=0}^{N-1}$ is PE of order $L+n$, then, for any $\Sigma\in\bm{\Sigma}_{n,m,p}^\rmc$ and any $x_0\in\bbr^n$,
\begin{align}\label{eq:thm_willems}
&\quad\mathrm{im}\begin{bmatrix}H_L(u)\\H_L(y(\Sigma,x_0,u))\end{bmatrix}\\\nonumber
&=
\Big\{\begin{bmatrix}\bar{u}\\\bar{y}(\Sigma,\bar{x}_0,\bar{u})\end{bmatrix}\mid \bar{u}\in\bbr^{mL},\bar{x}_0\in\bbr^n\Big\}.
\end{align}
\end{theorem}
Theorem~\ref{thm:willems} provides a parametrization of all system trajectories based on one input-output trajectory, and the result has found wide use in designing and analyzing data-driven control algorithms~\cite{markovsky2021behavioral}.
In the original paper~\cite{willems2005note}, the Fundamental Lemma has been proven in the behavioral framework, and a proof for state-space systems has been given in~\cite{waarde2020willems}.
In the following, we provide a novel, quantitative and constructive proof of Theorem~\ref{thm:willems}, which exploits the system description~\eqref{eq:sys_LTI_IO}.
To this end, we define the extended output $\phi_k=\begin{bmatrix}x_k\\u_k\end{bmatrix}$ which gives rise to the system
\begin{align}\label{eq:sys_LTI_tilde}
x_{k+1}&=Ax_k+Bu_k,\\\nonumber
\phi_k&=\tilde{C}x_k+\tilde{D}u_k,
\end{align}
where $\tilde{C}=\begin{bmatrix}I_n\\0\end{bmatrix}$, $\tilde{D}=\begin{bmatrix}0\\I_m\end{bmatrix}$.
Throughout this section, whenever we apply the results from Section~\ref{sec:pe_lin}, the considered ``output'' is $\phi_k$ and the corresponding ``input-output representation''~\eqref{eq:sys_LTI_IO} is associated with~\eqref{eq:sys_LTI_tilde}.
Note that, for system~\eqref{eq:sys_LTI_tilde}, the Markov parameters $\Gamma$ take the form $\Gamma=\begin{bmatrix}0&\Gamma_{\rmx\rmu}\\I&0\end{bmatrix}$ with the controllability matrix $\Gamma_{\rmx\rmu}=\begin{bmatrix}B&AB&\dots&A^{n-1}B\end{bmatrix}$.
Thus, system~\eqref{eq:sys_LTI_tilde} is output reachable if and only if $(A,B)$ is controllable.

To prove the Fundamental Lemma, we are only concerned with the input-state system~\eqref{eq:sys_LTI_tilde} and consider the matrix
\begin{align}\label{eq:Hux_def}
\begin{bmatrix}H_1(x_{[0,N-L]})\\H_L(u)\end{bmatrix}
=&\begin{bmatrix}x_0&x_1&\dots&x_{N-L}\\
u_0&u_1&\dots&u_{N-L}\\
u_1&u_2&\dots&u_{N-L+1}\\
\vdots&\vdots&\ddots&\vdots\\
u_{L-1}&u_L&\dots&u_{N-1}
\end{bmatrix}.
\end{align}
Specifically, we study under which conditions on the input the rank condition
\begin{align}\label{eq:rank_ux_condition}
\mathrm{rank}\left(\begin{bmatrix}H_1(x_{[0,N-L]})\\H_L(u)\end{bmatrix}\right)=mL+n,
\end{align}
holds, compare~\cite[Corollary 2 (iii)]{willems2005note}.
In the language of Definition~\ref{def:pe}, this means that $\left\{\begin{bmatrix}x_k\\u_{[k,k+L-1]}\end{bmatrix}\right\}_{k=0}^{N-L}$ needs to be PE of order $1$.
Establishing~\eqref{eq:rank_ux_condition} is the hard part of proving the Fundamental Lemma and, once shown, directly implies the statement of Theorem~\ref{thm:willems}, cf.\cite[Lemma 2]{persis2020formulas}.
We summarize this fact in the following statement.
\begin{proposition}\label{prop:WFL_intermediate}
    \eqref{eq:thm_willems} holds if~\eqref{eq:rank_ux_condition} holds.
\end{proposition}
In the following,
we derive a sufficient condition for~\eqref{eq:rank_ux_condition} to hold for arbitrary initial conditions and arbitrary controllable systems, which, according to Proposition~\ref{prop:WFL_intermediate}, provides a sufficient condition for~\eqref{eq:thm_willems}.
To this end, with $M$ corresponding to system~\eqref{eq:sys_LTI_tilde}, we define
\begin{align}\label{eq:thm_WFL_Z_def}
&Z=\left[\begin{array}{cc}
M&0_{(m+n)\times m(L-1)}\\\hline
0_{m(L-1)\times m}&T\otimes I_m
\end{array}\right]\\\nonumber
&
=\begin{bmatrix}M_n&M_{n-1}&\dots&M_1&M_0&0&\dots&0\\
0&d_nI&d_{n-1}I&\dots&d_1I&d_0I&\ddots&\vdots\\
\vdots&\ddots&\ddots&\ddots&\ddots&\ddots&\ddots&0\\
0&\dots&0&d_nI&d_{n-1}I&\dots&d_1I&d_0I
\end{bmatrix},
\end{align}
where $T$ is the Toeplitz matrix
\begin{align}\label{eq:T_Toeplitz}
T=
\begin{bmatrix}d_n&\dots&d_0&0&0\\
0&\ddots&\ddots&\ddots&0\\
0&0&d_n&\dots&d_0
\end{bmatrix}\in\bbr^{(L-1)\times(L+n-1)}.
\end{align} 
\begin{theorem}\label{thm:WFL}
    If $\{u_k\}_{k=0}^{N-1}$ is PE of order $L+n$, then,
for any $\Sigma\in\bm{\Sigma}_{n,m,p}^{\rmc}$ and any $x_0\in\bbr^n$,
\eqref{eq:thm_willems} holds.
Further, if $\{u_k\}_{k=0}^{N-1}$ is $K_\rmu$-PE of order $L+n$, then
\begin{align}\label{eq:thm_WFL_bound}
\begin{bmatrix}H_1(x_{[0,N-L]})\\H_L(u)\end{bmatrix}
\begin{bmatrix}H_1(x_{[0,N-L]})\\H_L(u)\end{bmatrix}^\top
\succeq\frac{1}{n+1}ZK_\rmu Z^\top.
\end{align}
\end{theorem}
\begin{proof}
Define
\begin{align*}
\Phi_k&\coloneqq\begin{bmatrix}\phi_{k-n}&\dots&\phi_k\end{bmatrix}
=\begin{bmatrix}x_{k-n}&\dots&x_k\\u_{k-n}&\dots&u_k\end{bmatrix},\\
U_k&\coloneqq \begin{bmatrix}
u_{k-n+1}&u_{k-n+2}&\dots&u_{k+1}\\
u_{k-n+2}&u_{k-n+3}&\dots&u_{k+2}\\
\vdots&\vdots&\ddots&\vdots\\
u_{k-n+L-1}&u_{k-n+L}&\dots&u_{k+L-1}\end{bmatrix}.
\end{align*}
First, note that
\begin{align*}
U_k d&=\begin{bmatrix}d_nu_{k-n+1}+\dots+d_0u_{k+1}\\
\vdots\\
d_nu_{k-n+L-1}+\dots+d_0u_{k+L-1}
\end{bmatrix}\\
&=(T\otimes I_m)u_{[k-n+1,k+L-1]}
\end{align*}
with $T$ from~\eqref{eq:T_Toeplitz}.
Hence, using~\eqref{eq:sys_LTI_IO}, we have
\begin{align}\label{eq:thm_WFL_proof_assertion}
\begin{bmatrix}\Phi_k\\
U_k\end{bmatrix}d&=
\begin{bmatrix}M&0\\0&T\otimes I_m\end{bmatrix}\begin{bmatrix}u_{[k-n,k]}\\u_{[k-n+1,k+L-1]}\end{bmatrix}\\\nonumber
&=Zu_{[k-n,k+L-1]}.
\end{align}
Using that $\{u_k\}_{k=0}^{N-1}$ is $K_\rmu$-PE of order $L+n$, we obtain
\begin{align}
ZK_\rmu Z^\top\preceq&\sum_{k=n}^{N-L}Zu_{[k-n,k+L-1]}u_{[k-n,k+L-1]}^\top Z^\top\\\nonumber
\stackrel{\eqref{eq:thm_WFL_proof_assertion}}{=}
&\sum_{k=n}^{N-L}\begin{bmatrix}\Phi_k\\U_k\end{bmatrix}dd^\top\begin{bmatrix}\Phi_k\\U_k\end{bmatrix}^\top
\preceq\lVert d\rVert_2^2\sum_{k=n}^{N-L}\begin{bmatrix}\Phi_k\\U_k\end{bmatrix}\begin{bmatrix}\Phi_k\\U_k\end{bmatrix}^\top\\\nonumber
\preceq&(n+1)\sum_{k=0}^{N-L}
\begin{bmatrix}x_k\\u_{[k,k+L-1]}\end{bmatrix}
\begin{bmatrix}x_k\\u_{[k,k+L-1]}\end{bmatrix}^\top,
\end{align}
which shows~\eqref{eq:thm_WFL_bound}.
Note that $Z$ has full row rank since $M$ has full row rank (by controllability) and $d_0\neq0$.
Together with Proposition~\ref{prop:WFL_intermediate}, this proves the statement.
\end{proof}

Theorem~\ref{thm:WFL} shows that (a variation of) the main result from Section~\ref{sec:pe_lin} can be used to prove the Fundamental Lemma from~\cite{willems2005note}.
In fact, for $L=1$, Theorem~\ref{thm:WFL} follows from Corollary~\ref{cor:PE_u_y} without any modification, i.e., the one-step version of the Fundamental Lemma~\cite[Corollary 2 (ii)]{willems2005note} is actually a direct consequence of~\cite[Corollary 2.1]{green1986persistence}.
The longer proof above is required for the general case $L>1$.

In addition to providing a novel proof strategy, Theorem~\ref{thm:WFL} in fact generalizes the original Fundamental Lemma from~\cite{willems2005note}
by providing a \emph{quantitative and constructive} statement.
Specifically,~\eqref{eq:thm_WFL_bound} reveals a direct connection between the PE level of the input and that of the input-state matrix $\begin{bmatrix}H_1(x_{[0,N-L]})\\H_L(u)\end{bmatrix}$.
While positive definiteness of the left-hand side in~\eqref{eq:thm_WFL_bound} (i.e., the existence of \emph{some} lower bound) is also guaranteed in existing proofs of the Fundamental Lemma~\cite{willems2005note,waarde2020willems}, we construct an explicit lower bound on the PE level.
Improved PE bounds, in turn, imply better robustness with respect to noise, e.g., in system identification~\cite{coulson2022robust} or in data-driven predictive control~\cite[Theorem 3]{berberich2021guarantees}.

These insights are of practical use even in the realistic scenario that the system parameters in $Z$ are only known approximately or if they are fully unknown.
To be precise, suppose only an estimate $\hat{Z}$ of $Z$ in~\eqref{eq:thm_WFL_Z_def} is available with error bound $\lVert\hat{Z}-Z\rVert_2<\varepsilon$ for some $\varepsilon>0$, and we want to use this knowledge to guarantee a certain PE bound in~\eqref{eq:thm_WFL_bound}.
If we let $\{u_k\}_{k=0}^{N-1}$ be $K_\rmu$-PE of order $L+n$, then the lower bound in~\eqref{eq:thm_WFL_bound} implies
\begin{align*}
&\begin{bmatrix}H_1(x_{[0,N-L]})\\H_L(u)\end{bmatrix}
\begin{bmatrix}H_1(x_{[0,N-L]})\\H_L(u)\end{bmatrix}^\top
\succeq\frac{1}{n+1}ZK_\rmu Z^\top\\
\succeq&\frac{1}{2(n+1)}\hat{Z}K_\rmu\hat{Z}^\top-\frac{\varepsilon^2}{n+1}\sigma_{\max}(K_\rmu).
\end{align*}
In this sense, our PE characterization is \emph{robust} with respect to model inaccuracy.
Similar bounds can be derived in the presence of process or measurement noise.
Working out these insights in more detail and employing them for robust data-driven control is an interesting issue for future research.

We note that related results (Fundamental Lemma with quantitative PE bounds) are provided in~\cite{coulson2022robust} using a proof strategy similar to~\cite{waarde2020willems}, however, only diagonal PE bounds and $L=1$ are addressed.

\section{Conclusion}
We presented a novel, quantitative and constructive proof of Willems' Fundamental Lemma, which has received significant attention in the context of data-driven control.
Additionally, we extended classical results on adaptive control, which inspired our proof.
We believe that the proposed new proof technique paves the way for further generalizations of the Fundamental Lemma and, since it enables robust PE guarantees for noisy/uncertain setups, it opens the door to new robust data-driven control results.


\bibliographystyle{IEEEtran}
\bibliography{Literature}

\renewcommand\thesection{\Alph{section}}
\setcounter{section}{1}

\section*{Appendix: On necessity in Theorem~\ref{thm:PE_Nu_y}}
\cite[Theorem 2.1]{green1986persistence} contains the following claim.
\begin{claim}\label{cl:app}
    If $\{y_k(\Sigma_0,x_0,u)\}_{k=0}^{N-1}$ is PE of order $1$ for any $x_0\in\bbr^n$, then
$\{Mu_{[k-n,k]}\}_{k=n}^{N-1}$ is PE of order $1$.
\end{claim}
Claim~\ref{cl:app} means that PE of order $1$ of $\{Mu_{[k-n,k]}\}_{k=n}^{N-1}$ is not only sufficient but also
\emph{necessary} for $\{y_k(\Sigma_0,x_0,u)\}_{k=0}^{N-1}$ being PE of order $1$ for arbitrary initial conditions.
However, as we show in the following, the claim is incorrect in general.

Suppose the system $\Sigma_0$ takes the form
\begin{align}\label{eq:app_counterexample}
    A=\begin{bmatrix}0&0\\1&0
    \end{bmatrix},\>
    B=\begin{bmatrix}1\\0\end{bmatrix},\>
    C=I_2,\>D=0.
\end{align}
With the minimal polynomial $A^2=0$, we have $
    d=\begin{bmatrix}0&0&1\end{bmatrix}^\top$,
    $M=\begin{bmatrix}0&1&0\\1&0&0
    \end{bmatrix}$.
Suppose now the input is chosen as $
    u_{[0,N-1]}=\begin{bmatrix}1&0&\dots&0\end{bmatrix}^\top$.
For a fixed initial condition $x_0=\begin{bmatrix}
x_0(1)\\x_0(2)
\end{bmatrix}$, the resulting output is
\begin{align}\label{eq:app_proof3}
    H_1(y)=&\begin{bmatrix}y_0&y_1&y_2&y_3&\dots&y_{N-1}
    \end{bmatrix}
    \\\nonumber
    =&\begin{bmatrix}
    x_0(1)&1&0&0&\dots&0\\
    x_0(2)&x_0(1)&1&0&\dots&0
    \end{bmatrix}.
\end{align}
Note that $H_1(y)$ has full row rank for any initial condition $x_0\in\bbr^2$.
On the other hand, $\{Mu_{[k-n,k]}\}_{k=n}^{N-1}$ is not PE of order $1$ since $Mu_{[0,2]}=\begin{bmatrix}
0\\1\end{bmatrix}$ but $Mu_{[k-n,k]}=0$ for all $k=n+1,\dots,N-1$.
This invalidates Claim~\ref{cl:app}.

Let us take a closer look at the proof of necessity in~\cite[Theorem 2.1]{green1986persistence}, i.e., an attempt of proof for Claim~\ref{cl:app}.
To this end, suppose $\{Mu_{[k-n,k]}\}_{k=n}^{N-1}$ is not PE of order $1$, i.e., there exists $0\neq a\in\bbr^p$ such that 
$a^\top Mu_{[k-n,k]}=0$ for $k=n,\dots,N-1$.
By~\eqref{eq:sys_LTI_IO}, this implies for $k=n,\dots,N-1$
\begin{align}\label{eq:app_proof}
    \sum_{j=0}^nd_ja^\top y_{k-j}=0.
\end{align}
The main proof idea from~\cite{green1986persistence} is now to choose initial conditions such that $a^\top y_k=0$ for $k=0,\dots,n-1$.
If this were possible, then one could repeatedly apply~\eqref{eq:app_proof} to conclude that $a^\top y_k=0$ for all $k=n,\dots,N-1$, thus proving
that the output is not PE of order $1$, i.e., proving Claim~\ref{cl:app}.
In~\cite{green1986persistence}, it is stated that such initial conditions can be chosen based on an observable realization of the
transfer function from $u$ to $a^\top y$.
Indeed, suppose $(\tilde{A},\tilde{B},\tilde{C},\tilde{D})$ is such a realization of order $\tilde{n}$, and denote the corresponding
state by $\tilde{x}_k\in\bbr^{\tilde{n}}$ and the output by $\tilde{y}_k=a^\top y_k\in\bbr$.
The system dynamics imply
\begin{align}\label{eq:app_proof2}
    \tilde{y}_{[0,\tilde{n}-1]}=\Phi_{\tilde{n}}\tilde{x}_0+\calT_{\tilde{n}}u_{[0,\tilde{n}-1]}
\end{align}
for the observability matrix $\Phi_{\tilde{n}}$ and a suitable matrix $\calT_{\tilde{n}}$.
Since $\Phi_{\tilde{n}}$ has full column rank (by observability) and is square (since $\tilde{y}_k\in\bbr$), it is invertible.
Hence, choosing the initial condition $\tilde{x}_0=-\Phi_{\tilde{n}}^{-1}\calT_{\tilde{n}}u_{[0,\tilde{n}-1]}$ leads to 
$\tilde{y}_k=a^\top y_k=0$ for $k=0,\dots,\tilde{n}-1$.

The problem arises from the fact that $\tilde{n}$ need not be equal to $n$.
In particular, if $\tilde{n}<n$, there may not exist initial conditions
which ensure that $a^\top y_k=0$ holds for $k=\tilde{n},\dots,n-1$.
Indeed, for the system~\eqref{eq:app_counterexample}, we have $a=\begin{bmatrix}1\\0
\end{bmatrix}$ and the system $(A,B,a^\top C,a^\top D)$ is not observable.
The state of an observable realization necessarily has dimension $\tilde{n}=1<2=n$.
This means that we can only choose initial conditions to ensure $a^\top y_0=0$, but not $a^\top y_1=0$, cf.~\eqref{eq:app_proof3}.

In view of the above argument, the proof of necessity from~\cite{green1986persistence} applies and Claim~\ref{cl:app} holds if
$\tilde{n}=n$, i.e., $(A,a^\top C)$ is observable.
This gives rise to the following open questions:
Are there relevant system classes for which observability of
$(A,a^\top C)$ can be guaranteed for any $a$ satisfying $a^\top Mu_{[k-n,k]}$ for $k=n,\dots,N-1$, independent of the choice of input?
Conversely, for which systems is $(A,a^\top C)$ not observable, i.e., $\tilde{n}<n$, such that smaller levels of PE for the input may suffice to achieve a certain PE level for the output?
What are the implications for necessity of PE in Corollary~\ref{cor:PE_u_y} and in the Fundamental Lemma (Theorem~\ref{thm:WFL})?
We leave these questions for future research.

\end{document}